\documentclass[11pt,reqno]{amsart}

\usepackage{amssymb}
\usepackage{mathrsfs}
\usepackage{cite}
\usepackage{amsfonts}
\usepackage{xspace}
\usepackage[all]{xy}
\CompileMatrices

\renewcommand{\to}{\longrightarrow}

\def\wh{\widehat}

\def\inv{^{-1}}
\def\p{\varphi}

\def\pinv{{\p \inv}}
\def\e<{\leq _{E}}

\def\ov#1{\ensuremath{\overline {#1}}}

\def\til#1{\ensuremath{\widetilde {#1}}}

\def\malce{\mathbin{\hbox{$\bigcirc$\rlap{\kern-8.3pt\raise0,50pt\hbox{$\mathtt{m}$}}}}}
\def\CC{\mathbb C}

\def\1sk{^{(1)}}

\def\gpd{G\ltimes \wh E}

\def\Thmname{Theorem}
\def\Propname{Proposition}
\def\Lemmaname{Lemma}
\def\Definitionname{Definition}
\newtheorem{Thm}{\Thmname}[section]
\newtheorem{Prop}[Thm]{\Propname}
\newtheorem{Lemma}[Thm]{\Lemmaname}
{\theoremstyle{definition}
\newtheorem{Def}[Thm]{\Definitionname}}
{\theoremstyle{remark}
}
\newtheorem{Cor}[Thm]{Corollary}

{\theoremstyle{remark}
\newtheorem{claim}{Claim}}

\numberwithin{equation}{section}

\title{On inverse semigroup $C^*$-algebras and crossed products}

\author{David Milan}
\address{Department of Mathematics\\
The University of Texas at Tyler\\
3900 University Blvd.\\
Tyler, Tx 75799}
\email{dmilan@uttyler.edu}

\author{Benjamin Steinberg}
\address{School of Mathematics and Statistics\\
Carleton University \\
1125 Colonel By Drive\\
Ottawa, Ontario  K1S 5B6 \\
Canada}
\thanks{The second author was supported in part by NSERC}
\email{bsteinbg@math.carleton.ca}
\date{\today}

\begin{document}
\begin{abstract}
We describe the $C^*$-algebra of an $E$-unitary or strongly $0$-$E$-unitary inverse semigroup as the partial crossed product of a commutative $C^*$-algebra by the maximal group image of the inverse semigroup. We give a similar result for the $C^*$-algebra of the tight groupoid of an inverse semigroup. We also study conditions on a groupoid $C^*$-algebra to be Morita equivalent to a full crossed product of a commutative $C^*$-algebra with an inverse semigroup, generalizing results of Khoshkam and Skandalis for crossed products with groups.
\end{abstract}
\maketitle

\section{Introduction}
$E$-unitary inverse semigroups are perhaps the most commonly studied class of inverse semigroups. One reason for this is that many interesting semigroups, including those studied in operator theory, are either $E$-unitary or strongly $0$-$E$-unitary. Another reason is that there is a very explicit structure theorem for this class known as McAlister's $P$-Theorem. It describes $E$-unitary inverse semigroups in terms of a group acting partially on a semilattice. It is natural to try to interpret this result at the level of $C^*$-algebras using crossed products, and indeed many authors have given descriptions of inverse semigroup algebras that are suggestive of this approach (c.f.~\cite{Skandalis,Nica,ExelLacaQuigg}), but none have been applicable in the same generality as the $P$-Theorem.

For example, in~\cite{Skandalis}, Khoshkam and Skandalis studied the $C^*$-algebras of certain locally compact groupoids admitting cocyles. They showed the algebras of such groupoids are Morita equivalent to crossed products. As an application to inverse semigroups, they show that the $C^*$-algebras of a restricted class of $E$-unitary inverse semigroups are Morita equivalent to crossed products of the maximal group image by a commutative $C^*$-algebra related to the idempotents. They also point out that their results do not hold for all $E$-unitary inverse semigroups.

In Section~3 of this paper it is shown that the $C^*$-algebra of an $E$-unitary inverse semigroup is isomorphic to a partial crossed product (in the sense of~\cite{ExelParCross} and~\cite{McClanahanParCross}) of the maximal group image $G$ by the algebra of its subsemigroup of idempotents $E$. In fact, results are given for both the full and reduced $C^*$-algebras that correspond to full and reduced partial crossed products respectively.

Next we use groupoid reductions to apply the results of Section~3 to the inverse semigroups that typically appear in the $C^*$-algebra literature. Most contain a zero element, but inverse semigroups with zero are only $E$-unitary in trivial cases. However, in Section~5 we give crossed product results for strongly $0$-$E$-unitary inverse semigroups, which are precisely the inverse semigroups with zero for which a variant of the $P$-Theorem holds. By considering reductions we are also able to give crossed product results for the $C^*$-algebra of the tight groupoid associated with such an inverse semigroup in~\cite{Exel}. This is especially important in light of examples such as the Cuntz algebras $\mathcal{O}_n$ and the $C^*$-algebras of directed graphs that are quotients of inverse semigroup algebras, and have also been identified as the $C^*$-algebras of the tight groupoid of the relevant inverse semigroups~\cite{Exel}.

In the final section we study Morita equivalence with a view toward understanding the relationship between our results, the work of Khoshkam and Skandalis~\cite{Skandalis}, and related work of Abadie~\cite{abadie,abadie2}. We find that the Morita equivalence in~\cite{abadie} can be viewed as a special case of some of the results in~\cite{Skandalis}. Also, by replacing the cocycle appearing in~\cite{Skandalis} with a morphism between groupoids, we study Morita equivalence in a more general context.  In particular, we give conditions which guarantee that a locally compact groupoid is Morita equivalent to a groupoid of germs of an inverse semigroup action.  To apply this to inverse semigroups,  we investigate the functoriality of the assignment $S \mapsto \mathscr G(S)$ of Paterson's universal groupoid to an inverse semigroup $S$~\cite{Paterson}. We define a condition, called the Khoshkam-Skandalis condition, on a morphism of inverse semigroups that generalizes the one in Proposition~3.9 of~\cite{Skandalis}. It guarantees that a morphism $\p\colon S\to T$ induces a Morita equivalence between $C^*(S)$ and a crossed product of $T$ by a commutative $C^*$-algebra, generalizing the Morita equivalence in Corollary~3.11 of~\cite{Skandalis} for the case where $T$ is a group and $\p$ is the maximal group image homomorphism.  In particular, we generalize Theorem~3.10 of~\cite{Skandalis} from $E$-unitary inverse semigroups to strongly $0$-$E$-unitary inverse semigroups and locally $E$-unitary inverse semigroups. A key role is played by the identification of a certain category of $T$-actions with the category of $\mathscr G(T)$-actions.

\section{Preliminaries}
In this paper, compact will always mean compact and Hausdorff. We will use the term \emph{quasi-compactness} for the condition that every open cover has a finite subcover.

A groupoid $\mathscr G$ is a small category in which each arrow is invertible.  We will use the arrows only approach to groupoids and identify the objects with the space of units $\mathscr G^0$ when convenient.  The domain and range of an arrow are denoted $d$ and $r$ respectively.  For more on topological groupoids, see~\cite{renault,Paterson}.  Following Paterson, we assume that a left Haar system is part of the definition of a locally compact groupoid~\cite{Paterson}.

An \emph{\'etale groupoid} is a topological groupoid $\mathscr G$ whose unit space  is locally compact Hausdorff, and such that the domain map $d\colon \mathscr G \to \mathscr G^0$ (or equivalently the range map) is a local homeomorphism. We do not assume $\mathscr G$ is Hausdorff. The counting measures give a left Haar system for an \'etale groupoid~\cite{Exel,Paterson,renault}.

A semigroup $S$ is an \textit{inverse semigroup} if, for each $s$ in $S$, there exists unique $s^*$ in $S$ such that \[s = s s^* s\quad \text{and}\quad s^* = s^* s s^*.\]
We state some basic definitions concerning inverse semigroups here; a thorough treatment of the subject can be found in~\cite{Lawson}.

There is a natural partial order on $S$ defined by $s \leq t$ if $s = te$ for some idempotent $e$.  Equivalent conditions are the following: $s=ft$ for some idempotent $f$; $s=ts^*s$; and $s=ss^*t$. The subsemigroup $E(S)$ of idempotents of $S$ is commutative, and hence forms a (meet) semilattice for the natural partial order where $e \wedge f= ef$ for $e,f$ in $E(S)$. Every inverse semigroup admits a universal morphism $\sigma\colon S\to G(S)$ onto a group. The group $G(S)$ is the quotient of $S$ by the congruence $s \sim t$ if and only if $se = te$ for some $e$ in $E(S)$. It is universal in the sense that all other morphisms from $S$ to a group factor uniquely through $\sigma$. One calls $G(S)$ the \emph{maximal group image} of $S$ and $\sigma$ the \emph{maximal group homomorphism}. The inverse semigroup $S$ is called \emph{$E$-unitary} if $\sigma^{-1}(1) = E(S)$. An equivalent condition to being $E$-unitary is that $s^*s=t^*t$ and $\sigma(s)=\sigma(t)$ implies $s=t$.

Given a locally compact Hausdorff space $X$, denote by $I_X$ the inverse monoid of all homeomorphisms between open subsets of $X$, with multiplication given by composition on the largest domain where it is defined.  The natural partial order on $I_X$ is given by $f\leq g$ if and only if $f$ is a restriction of $g$ to some open subset of $X$.

A map $\theta\colon S\to T$ of inverse semigroups is called a \emph{dual prehomomorphism} if $\theta(s)\theta(s')\leq \theta(ss')$ all $s,s'\in S$. Let $G$ be a countable discrete group. A \emph{partial action} of $G$ on $X$ is a dual prehomomorphism $\theta\colon G\to I_X$ of inverse monoids (so, in particular, $\theta(1)=1_X$); see~\cite{Lawsonkellendonk} for the equivalence of this definition with the one in, say~\cite{abadie}.  We write $X_{g\inv}$ for the domain of $\theta(g)$ and if $x\in X_{g\inv}$, then we write $gx$ for $\theta(g)(x)$ (which is an element of $X_g$). Given a partial action of $G$ on $X$, one can form a Hausdorff \'etale groupoid \[G\ltimes X=\{(g,x)\mid g\in G, x\in X_{g\inv}\}\]  with the subspace topology of the product topology.  We identify $X$ with $\{1\}\times X$, which will be the unit space.  The domain and range maps are given by $d(g,x)=x$ and $r(g,x) = gx$.  The product is defined by $(g,x)(h,y) = (gh,y)$, when $x=hy$.   The inverse is given by $(g,x)\inv = (g\inv,gx)$.  This construction appears in Abadie~\cite{abadie2}, but was considered independently by Lawson and Kellendonk~\cite{Lawsonkellendonk} in the discrete setting.  We call $G\ltimes X$ the \emph{partial transformation groupoid} of the partial action $\theta$.

Abadie~\cite{abadie2} proved the following result.

\begin{Thm}[Abadie]\label{abadie1}
Let $\theta\colon G\to I_X$ be a partial action of a countable discrete group on a locally compact Hausdorff topological space $X$.    Then the universal $C^*$-algebra of $G\ltimes X$ is isomorphic to the universal partial action cross product $C_0(X)\rtimes G$.
\end{Thm}

An \emph{action} of an inverse semigroup $S$ on $X$ is a morphism $\theta\colon S\to I_X$ such that the union of the domains of $\theta_s$ with $s\in S$ is $X$. We write $X_e$ for the domain of $\theta_e$ for an idempotent $e\in E(S)$. From an inverse semigroup action one can define the \emph{groupoid of germs}, which is similar to the partial transformation groupoid defined above. Let \[\Omega = \{(s,x) \in S \times X \mid x \in X_{s^*s} \},\]
and define an equivalence relation on $\Omega$ by $(s,x) \sim (t,y)$ if $x=y$ and there exists $e$ in $E(S)$ such that $x \in X_e$ and $se = te$. Equivalently, $s\sim t$ if $x=y$ and there exists $u\leq s,t$ such that $x\in X_{u^*u}$.  The class of $(s,x)$ is denoted by $[s,x]$.

The \emph{groupoid of germs} $S\ltimes X$ of the action $\theta$ is the set $\Omega/{\sim}$ with multiplication given by $[s,x][t,y] = [st, y]$ provided $x = \theta_t(y)$. The inverse of $[s,x]$ is $[s^*, \theta_s(x)]$. The topology has basis consisting of all sets of the form \[(s,U)=\{[s,x]\mid x\in U\}\] where $s\in S$ and $U$ is an open subset of $X$. For a detailed construction of this groupoid see~\cite{Exel}.

A \emph{semi-character} of a semilattice $E$ is a non-zero semilattice homomorphism $\p\colon E\to \{0,1\}$.
The space $\wh E$ (topologized as a subspace of $\{0,1\}^E$) of semi-characters is fundamental because $C^*(E) \cong C_0(\wh E)$. Every inverse semigroup $S$ acts on its space of semi-characters.  Putting $E=E(S)$, for $e\in E$, let \[D(e) = \{\p\in \wh E\mid \p(e)=1\}.\] It is a clopen subset of $\wh E$ and the sets of the form $D(e)$ and their complements are a subbasis for the topology on $\wh E$.  Define, for each $s\in S$ a mapping $\beta_s\colon D(s^*s)\to D(ss^*)$ by  $\beta_s(\p)(e) = \p(s^* e s)$ for each $e \in E$, $\p \in D(s^*s)$.  Then $\beta\colon S\to I_{\wh E}$ given by $\beta(s)=\beta_s$ is an action.  Usually, we write $s\p$ for $\beta_s(\p)$.
Paterson's universal groupoid $\mathscr{G}(S)$ is the groupoid of germs of the action $\beta$~\cite{Paterson,Exel,discretegroupoid}. Paterson~\cite{Paterson} proved that $C^*(S) \cong C^*(\mathscr{G}(S))$ and $C_{r}^*(S) \cong C_{r}^*(\mathscr{G}(S))$.

\section{The $C^*$-algebra of an $E$-unitary inverse semigroup}
Let $S$ be a countable $E$-unitary inverse semigroup with idempotent set $E$ and maximal group image $G = G(S)$. Our goal in this section is to show that $C^*(S)$ and $C^*_r(S)$ are partial action cross products $C^*(E)\rtimes G$ and $C^*(E)\rtimes_r G$, respectively.  The first author in his thesis gave a direct construction of a partial action of $G$ on $C^*(E)$ and showed the corresponding cross product algebra was isomorphic to the algebra $C^*(S)$.  We use here the theory of partial actions of groups on topological spaces. In particular, we show that the universal groupoid $\mathscr G(S)$ of $S$ is isomorphic to the partial transformation groupoid $G\ltimes \wh E$ of an appropriate partial action of $G$ on $\wh E$.

We wish to construct a partial action $\theta\colon G\to I_{\wh E}$. It is well-known in inverse semigroup theory that $G$ acts partially on $E$: this is one of Lawson's interpretations of McAlister's $P$-theorem~\cite{Lawsonkellendonk}.  We extend the partial action to $\wh E$ to define $\theta$.  In some sense, this result is not unprecedented.  In fact, it is known that $\mathscr G(S)$ is algebraically the underlying groupoid of a certain inverse semigroup $T$ containing $S$~\cite{lenz}.  It is easy to verify that if $S$ is $E$-unitary, then $T$ is as well and they both have maximal group image $G$.  Therefore, $G$ acts partially on $E(T)=\wh E$ and the underlying groupoid of $T$ is the partial transformation groupoid $G\ltimes \wh E$; this is essentially the content of McAlister's $P$-theorem as interpreted via~\cite{Lawsonkellendonk}.  So what we need to do is make sure that everything works topologically.  However, we do not assume here any knowledge of the $P$-theorem or the inverse semigroup structure on $\mathscr G(S)$.

Define $\theta\colon G\to I_{\wh E}$ by setting \[\theta(g) = \bigcup_{s\in \sigma\inv(g)} \beta(s)\] where we are viewing partial functions as relations.   To show that $\theta(g)$ is a well-defined continuous map, we just need to show that if $\sigma(s)=\sigma(t)$  and $\p\in D(s^*s)\cap D(t^*t)$, then $s\p = t\p$ (where we drop $\beta$ from the notation).  To prove this, we need a well-known fact about $E$-unitary inverse semigroups.

\begin{Lemma}\label{meets}
Let $S$ be an $E$-unitary inverse semigroup and suppose $s,t\in S$ satisfy $\sigma(s)=\sigma(t)$.  Then $ts^*s=st^*t$ is the meet of $s,t$ in the natural partial order.  Moreover, if $u=ts^*s$, then $u^*u = s^*st^*t$.
\end{Lemma}
\begin{proof}
Let $u=ts^*s$ and $v=st^*t$.  Then $u^*u= s^*st^*ts^*s=s^*st^*t$, $v^*v= t^*ts^*st^*t=s^*st^*t$ and $\sigma(u) =\sigma(t)=\sigma(s)=\sigma(v)$.  Because $S$ is $E$-unitary, this implies $u=v$.  Clearly $u\leq t, v\leq s$ so $u\leq s,t$.  If $w\leq s,t$, then $uw^*w = st^*tw^*w = sw^*w=w$ so $w\leq u$.  This completes the proof.
\end{proof}

To obtain now that $\theta(g)$ is well defined, suppose that $\sigma(s)=\sigma(t)$ and $\p\in D(s^*s)\cap D(t^*t)$.  Setting $u=ts^*s=st^*t$, we have by Lemma~\ref{meets} that $D(u^*u) = D(s^*st^*t) = D(s^*s)\cap D(t^*t)$ and $u\leq s,t$.  Thus $\p\in D(u^*u)$ and $s\p=u\p=t\p$.   So $\theta(g)$ is a well-defined continuous function with open domain $X_{g\inv} = \bigcup_{s\in \sigma\inv (g)} D(s^*s)$.  From the definition, it is immediate that $\theta(g\inv) = \theta(g)\inv$ and so $\theta(g)\colon X_{g\inv}\to X_g$ is a homeomorphism, and hence belongs to $I_{\wh E}$.  Since $\bigcup_{e\in E} D(e) = \wh E$, it follows that $\theta(1)=1_X$.  Finally, we verify that $\theta$ is a dual prehomomorphism by noting that since $\sigma\inv(g)\sigma\inv(h)\subseteq \sigma\inv(gh)$, it follows that
\begin{align*}
\theta(g)\theta(h) &= \bigcup_{s\in \sigma\inv(g)}\beta(s)\cdot\bigcup_{t\in \sigma\inv(h)} \beta(t)= \bigcup_{x\in \sigma\inv(g)\sigma\inv(h)}\beta(x)\\ &\leq \bigcup_{s\in \sigma\inv(gh)}\beta(s) = \theta(gh).
\end{align*}
This proves that $\theta$ is a partial action.  From now on we suppress the notation $\beta,\theta$.
Next we prove that $G\ltimes \wh E$ is isomorphic to $\mathscr G(S)$ as a topological groupoid.  This gives a new proof that $\mathscr G(S)$ is Hausdorff for $E$-unitary inverse semigroups. An exact characterization of inverse semigroups with Hausdorff universal groupoids is given in~\cite[Theorem~5.17]{discretegroupoid}.

\begin{Thm}\label{main1}
Let $S$ be a countable $E$-unitary inverse semigroup with idempotent set $E$ and maximal group image $G$.  Then the universal groupoid $\mathscr G(S)$ is topologically isomorphic to $G\ltimes \wh E$.
\end{Thm}
\begin{proof}
Recall that arrows of $\mathscr G(S)$ are equivalence classes $[s,\p]$ of pairs $(s,\p)$ with $s\in S$ and $\p\in D(s^*s)$, where $(s,\p)\sim (t,\psi)$ if and only if $\p=\psi$ and there exists $u\in S$ with $u\leq s,t$ and $\p\in D(u^*u)$.  The topology has basis of open sets $(s,U)$ where $U\subseteq D(s^*s)$ is an open subset of $\wh E$ and $(s,U) = \{[s,\p]\mid \p\in U\}$.

Define functors $\Phi\colon \mathscr G(S)\to \gpd$ and $\Psi\colon \gpd \to \mathscr G(S)$ as follows.  Define $\Phi[s,\p] = (\sigma(s),\p)$ and $\Psi(g,\p) = [s,\p]$ where $s\in \sigma\inv (g)$ and $\p\in D(s^*s)$.  First we show that $\Phi$ and $\Psi$ are well-defined functors, beginning with $\Phi$.  As $(s,\p)\sim (t,\psi)$ implies $s$ and $t$ have a common lower bound, if $(s,\p)\sim (t,\psi)$, then $\p=\psi$ and $\sigma(s)=\sigma(t)$.  Moreover, $\p\in X_{g\inv}$ by construction.  Thus $\Phi$ is well-defined.  To see that $\Psi$ is well-defined, note that if $(g,\p)\in \gpd$, then by definition there exists $s\in \sigma\inv (g)$ so that $\p\in D(s^*s)$.  If $t\in \sigma\inv (g)$ so that $\p\in D(t^*t)$, then $u=ts^*s$ is a common lower bound for $s,t$ with $\p\in D(s^*s)\cap D(t^*t) = D(u^*u)$ by Lemma~\ref{meets}.  So $(s,\p)\sim(t,\p)$ and hence $\Psi$ is well-defined.  Notice that both $\Phi$ and $\Psi$ are the identity map on the unit space $\wh E$.

%Let us check they are functors.
It is routine to verify that $\Phi$ is a functor.
%Firstly, $d(\Phi[s,\p]) = d(\sigma(s),\p) = \p=\Phi(d[s,\p])$ and $r(\Phi[s,\p]) = \sigma(s)\p = %s\p = \Phi(r[s,\p])$.   Finally, if $[s,\p],[t,\psi]$ are elements of $\mathscr G(S)$ with $\p = %t\psi$, then
%\[\Phi [s,\p]\Phi[t,\psi] = (\sigma(s),\p)(\sigma(t),\psi) = (\sigma(st),\psi) = \Phi[st,\psi]\] %and so $\Phi$ is a functor.
Let us verify that $\Psi$ is a functor. Let $(g,\p)\in \gpd$ and choose $s\in \sigma\inv (g)$ with $\p\in D(s^*s)$.  Then $d(\Psi(g,\p)) = d[s,\p] = \p = \Psi(d(g,\p))$ and $r(\Psi(g,\p)) = r[s,\p] = s\p =g\p =\Psi(r(g,\p))$.  Moreover, if $(g,\p),(h,\psi)\in \gpd$ with $\p = h\psi$ and $s\in \sigma\inv (g)$, $t\in \sigma\inv (h)$ with $\p \in D(s^*s)$, $\psi \in D(t^*t)$, then
\[\Psi(g,\p)\Psi(h,\psi) = [s,\p][t,\psi] = [st,\p\psi] = \Psi(gh,\psi)\] since $\sigma(st) = gh$ and \[1=\p(s^*s) = h\psi(s^*s) = t\psi(s^*s) = \psi(t^*s^*st) = \psi((st)^*st).\]

Next observe that if $[s,\p]\in \mathscr G(S)$, then $\Psi\Phi[s,\p] = \Psi(\sigma(s),\p) = [s,\p]$ since $\p \in D(s^*s)$.  Also if $(g,\p)\in \gpd$ and $s\in \sigma^{-1}(g)$ with $\p \in D(s^*s)$, then $\Phi\Psi(g,\p) = \Phi[s,\p] = (g,\p)$.  Thus $\Phi$ and $\Psi$ are inverse functors.  To show they are homeomorphisms it suffices to show that if  $(s,U)$ is a basic open set of $\mathscr G(S)$, then $\Phi(s,U)$ is open and if $U$ is an open subset of $\wh E$ and $g\in G$, then $\Psi(\{g\}\times U)$ is open.  But $\Phi(s,U) = \{\sigma(s)\} \times U$, which is open, whereas \[\Psi(\{g\}\times U) = \bigcup_{s\in \sigma\inv(g)} (s,U\cap D(s^*s)),\] which again is open.  This completes the proof.
\end{proof}

We now state some corollaries.  The first is well-known~\cite{Paterson}.

\begin{Cor}
The universal groupoid of an $E$-unitary inverse semigroup is Hausdorff.
\end{Cor}

Since $C^*(\mathscr G(S))\cong C^*(S)$ we obtain a new proof of the following result from the first author's thesis by applying Theorem~\ref{abadie1}.

\begin{Cor}\label{coruniversal}
Let $S$ be a countable $E$-unitary inverse semigroup with idempotent set $E$ and maximal group image $G$.  Then $C^*(S)\cong C^*(E)\rtimes G\cong C_0(\wh E)\rtimes G$.
\end{Cor}

We will now show $C^*_r(S)\cong C^*(E)\rtimes_r G$. If Abadie's crossed product result extended to the reduced algebra of the partial transformation groupoid then the isomorphism would be immediate, since $C^*_r(\mathscr G(S))\cong C_r^*(S)$. The authors suspect that such a result is true, but in the absence of a proof, we prove the isomorphism directly. The proof makes use of the original construction of the reduced partial cross product $A \rtimes_r G$~ of a $C^*$-algebra $A$ by a discrete group $G$ \cite{McClanahanParCross}. McClanahan constructs a covariant representation $\til{\pi} \times \lambda \colon A \rtimes_r G \to B(H \otimes \ell^2(G))$ arising from a representation $\pi \colon A \to B(H)$ of $A$ on a Hilbert space $H$, where $\lambda$ is the left regular representation of $G$. If $\pi$ is faithful then $\til{\pi} \times \lambda$ is also faithful~\cite[Proposition~3.4]{McClanahanParCross}. Applying this to the left regular representation of the semilattice, $\pi \colon C^*(E) \to B(\ell^2(E))$, we get a faithful representation of  $C^*(E)\rtimes_r G$ on $B(\ell^2(E) \otimes \ell^2(G))$. Define, for $s\in S$, $F_s\colon G \to C^*(E)$ by
\[F_s(g) = \begin{cases} ss^* & \text{if}\ \sigma(s)=g\\ 0 & \text{else.}\end{cases}\]
The partial crossed product $C^*(E)\rtimes_r G$ is the closed span of the elements $F_s$~\cite{McClanahanParCross}. Moreover,

\[(\til{\pi} \times \lambda)(F_s)( \delta_e \otimes \delta_g ) = \pi(\theta_{g^{-1}\sigma(s^*)})\delta_e  \otimes \lambda_{\sigma(s)g}.\]

Note that the above equation reduces to $0$, unless there exists $t \in \sigma^{-1}(g)$ such that $e \leq t^*s^*st$. Define an operator $U: \ell^2(S)\to\ell^2(E) \otimes \ell^2(G)$ by $\delta_s \mapsto \delta_{s^*s} \otimes \delta_{\sigma(s)}$.

\begin{claim} Suppose for some $s \in S$ that $(\widetilde{\pi} \times \lambda)(F_s)(\delta_e \otimes \delta_g) \neq 0$. Then $\delta_e \otimes \delta_g \in \text{Ran}(U)$.
\end{claim}
\begin{proof} Suppose that $(\widetilde{\pi} \times \lambda)(F_s)(\delta_e \otimes \delta_g) \neq 0$. Then there exists $t \in \sigma^{-1}(g)$ such that $e \leq t^*s^*st$. In that case notice that $U(\delta_{te}) = \delta_{e} \otimes \delta_{g}$.

\end{proof}

\begin{Thm}\label{main1reduced} Let $S$ be a countable $E$-unitary inverse semigroup with idempotent set $E$ and maximal group image $G$.  Then $C^*_r(S)\cong C^*(E)\rtimes_r G$.
\end{Thm}
\begin{proof}
Note that the map $U$ defined above is injective on basis elements because $S$ is $E$-unitary. By the above claim, $U$ is a unitary operator from $\ell^2(S)$ to the essential subspace of $(\widetilde{\pi} \times \lambda)(C^*(E)\rtimes_r G)$. 

Recall that $C^*_r(S)$ may be viewed as a subalgebra of $B(\ell^2(S))$ via the left regular representation $\Lambda$, where
\[\Lambda_s \delta_t = \begin{cases} \delta_{st} & \text{if}\ s^*st = t \\
                        0  & \text{otherwise.}\end{cases} \]

We will show that $U$ intertwines the operators $\Lambda_s$ and $(\widetilde{\pi} \times \lambda)(F_s)$ and therefore implements a $*$-isomorphism from $C^*_r(S)$ to the faithful image of $C^*(E)\rtimes_r G$ inside $B(\ell^2(E) \otimes \ell^2(G))$. Given $s,t$ in $S$,
\[U \Lambda_s \delta_t = U \delta_{st} =  \delta_{(st)^*st} \otimes \delta_{\sigma(st)},\] provided  $s^* s t = t$.
On the other hand,
\begin{align*}
[(\til\pi\times \lambda)(F_s)] U\delta_t &= (\til\pi \times \lambda)(F_s)(\delta_{t^*t} \otimes \delta_{\sigma(t)})\\
&= \pi(\theta_{\sigma(t^*)\sigma(s^*)}(ss^*)) \delta_{t^*t} \otimes \delta_{\sigma(s)\sigma(t)} \\
&= \pi(t^*s^*st) \delta_{t^*t} \otimes \delta_{\sigma(s)\sigma(t)} \\
&= \delta_{t^*t} \otimes \delta_{\sigma(s)\sigma(t)}\quad (\text{provided}\  t^*t \leq t^*s^*st) \\
&= \delta_{(st)^* st} \otimes \delta_{\sigma(st)}. \\
\end{align*}

Notice that the conditions $s^*st = t$, $t^*t = t^*s^*st$, and $t^*t \leq  t^*s^*st$ are equivalent and that both operators are 0 when they fail to hold.
\end{proof}

We thank Alcides Buss for finding a gap in the proof of the above theorem in an earlier version of this paper.

\section{Ideal quotients and reductions}
Let $S$ be an inverse semigroup with zero element $z$.  By the \emph{contracted} universal and reduced $C^*$-algebras of $S$ we mean the quotient of the usual algebras by the one-dimensional, central, closed ideal $\CC z$.  Since we only will consider contracted algebras for inverse semigroups with zero, we use the notations $C^*(S)$ and $C^*_r(S)$ for the contracted $C^*$-algebras in this setting.  No confusion should arise.  These algebras encode $*$-representations of $S$ that send $z$ to zero.  From now on we identify $z$ with the $0$ of the algebra and drop the notation $z$.

Exel defines~\cite{Exel} the universal groupoid in this context.  If $E$ is a semilattice with zero, then we redefine $\wh E$ to be the space of all non-zero homomorphisms $\p\colon E\to \{0,1\}$ such that $\p(0)=0$.  This is a closed invariant subspace of the usual semi-character space.  This abuse of notation should not lead to any confusion.

If $\mathscr G$ is a groupoid and $X$ is a subspace of the unit space, then the \emph{reduction} $\mathscr G|_X$ of $\mathscr G$ to $X$ is the full subgroupoid of $\mathscr G$ with unit space $X$ and arrows $\{g\in \mathscr G\mid d(g),r(g)\in X\}$.  The universal groupoid $\mathscr G(S)$ of an inverse semigroup $S$ with zero is then the reduction of Paterson's universal groupoid to the space of semi-characters sending $0$ to $0$.  One can easily prove, cf.~\cite{Exel}, that $S$ and its universal groupoid have the same universal and reduced $C^*$-algebras in this context.

An \emph{ideal} of a semigroup $S$ is a subset $I$ such that $SI\cup IS\subseteq I$.  The \emph{Rees quotient} $S/I$ is the inverse semigroup obtained by identifying all elements of $I$ to a single element, which will be the zero of $S/I$.

\begin{Lemma}\label{ideal}
Let $S$ be an inverse semigroup (with or without zero) and idempotent set $E$.  Let $I$ be a proper ideal of $S$.  Let \[I^{\perp} = \{\p\in \wh E\mid \p(E\cap I) = 0\}.\]  Then $I^{\perp}$ is a closed subspace of $\wh E$ homeomorphic to $\wh{E(S/I)}$, which moreover is invariant under the action of $S$. The restriction of the action of $S$ to $I^{\perp}$ factors through $S/I$ as its usual action of $\wh{E(S/I)}$ (up to isomorphism).
\end{Lemma}
\begin{proof}
It is trivial that a non-zero homomorphism of semilattices with zeroes  $\p\colon E(S/I)\to \{0,1\}$ is the same thing as a morphism $\p\colon E(S)\to \{0,1\}$ vanishing on $I$ and that the topologies on $\wh {E(S/I)}$ and $I^{\perp}$ agree under these identifications.  Hence $I^{\perp}\cong \wh {E(S/I)}$.  Also $I^{\perp}$ is clearly closed.   Suppose that $\p \in D(s^*s)\cap I^{\perp}$.  Then $\p(s^*s)=1$ implies that $s^*s\notin I$ and hence $s\notin I$.    Let us check that $s\p \in I^{\perp}$.  Indeed, if $e\in E\cap I$, then $s\p(e) = \p(s^*es) =0$ since $s^*es\in E\cap I$ and $\p$ vanishes on $E\cap I$.  Thus $s\p\in I^{\perp}$.

Since we saw no element of $I$ is defined on $I^{\perp}$, it follows that the restricted action factors through $S/I$.  The action is clearly an isomorphic with the usual action of $S/I$ on $\wh {E(S/I)}$.
\end{proof}

The following proposition is an immediate consequence.

\begin{Prop}\label{reduction}
Let $S$ be an inverse semigroup and $I$ an ideal.  Then $\mathscr G(S/I)$ is topologically isomorphic to $\mathscr G(S)|_{I^{\perp}}$.
\end{Prop}

\section{Strongly $0$-$E$-unitary inverse semigroups}
If $S$ is an inverse semigroup with zero and $G$ is a group, then a \emph{partial homomorphism} $\p\colon S\to G$ is a mapping $\p\colon S\setminus \{0\}\to G$ such that $\p(st)=\p(s)\p(t)$ whenever $st\neq 0$.  Notice that the image of $\p$ need not be a subgroup of $G$.  If $e\in E(S)\setminus \{0\}$, then $\p(e)^2=\p(e)$ so $\p(e)=1$.   One calls a partial homomorphism $\p\colon S\to G$ \emph{idempotent pure} if $\p^{-1}(1) = E(S)$.  An inverse semigroup $S$ with zero is called \emph{$0$-$E$-unitary} if $s\geq e\neq 0$ with $e$ idempotent implies $s=s^2$.  It is called \emph{strongly $0$-$E$-unitary}~\cite{Lawson0eunit,Fountain} if it admits an idempotent pure partial homomorphism to a group (some authors use the term strongly $E^*$-unitary).  A strongly $0$-$E$-unitary inverse semigroup is $0$-$E$-unitary.

Margolis first observed that there is a universal group associated to an inverse semigroup $S$ with zero~\cite{undeceunit,Lawsoncoimbra}.  Let $S$ be an inverse semigroup with zero and define $\mathscr U(S)$ to be the group generated by the underlying set of $S$, subject to the relations $s\cdot t=st$ if $st\neq 0$.  Then the identity map on $S$ induces  a partial homomorphism $\iota\colon S\to \mathscr U(S)$ and all partial homomorphisms from $S$ to a group factor uniquely through this one.  Thus $S$ is strongly $0$-$E$-unitary if and only if $\iota$ is idempotent pure.  Note that if $S$ is finite, one can still have that $\mathscr U(S)$ is infinite so the map $\iota$ is by no means onto.  In fact, the second author showed it is undecidable whether a finite inverse semigroup is strongly $0$-$E$-unitary~\cite{undeceunit}.  If $S$ is an $E$-unitary inverse semigroup, then $S^0=S\cup \{0\}$ is easily verified to be strongly $0$-$E$-unitary with universal group the maximal group image $G(S)$ of $S$.

Nearly all the inverse semigroups studied in $C^*$-algebra theory are strongly $0$-$E$-unitary.  For instance, the Cuntz semigroup~\cite{Paterson}, also known as the polycyclic inverse monoid~\cite{Lawson}, on a set $X$ of cardinality at least $2$ is strongly $0$-$E$-unitary.  Recall that $P_X = \langle X\mid x^*y = \delta_{xy}\rangle$.  Each non-zero element of $P_X$ can be uniquely written $wu^*$ with $w,u$ positive words over $X$.  One has $\mathscr U(P_X)=F_X$, the free group on $X$, and the natural partial homomorphism $\iota\colon P_X\to F_X$ takes $wu^*$ to the reduced form of $wu\inv$.  The non-zero idempotents of $P_X$ are the elements of the form $uu^*$ and so $\iota$ is idempotent pure.  More generally, if $\Gamma$ is a connected graph (which is not just a single loop), then the graph inverse semigroup~\cite{Paterson} associated to $\Gamma$ is strongly $0$-$E$-unitary with universal group the fundamental group of the underlying graph.  Kellendonk's tiling inverse semigroups are also known to be strongly $0$-$E$-unitary~\cite{Tiling,Lawsoncoimbra}.

It was independently observed by Margolis, McAlister and the second author that strongly $0$-$E$-unitary inverse semigroups are precisely the Rees quotients of $E$-unitary inverse semigroups.  Namely, if $\theta\colon S\to G$ is an idempotent pure partial homomorphism such that $G$ is generated by the image of $\theta$ (one can always assume this), then $T=\langle (s,\theta(s))\rangle\subseteq S\times G$ is an $E$-unitary inverse semigroup with maximal group image $G$ and $I=(0\times G)\cap T$ is an ideal with $T/I\cong S$.  Conversely, if $T$ is an $E$-unitary inverse monoid with maximal group image $G$ and $I$ is an ideal, then $S=T/I$ is strongly $0$-$E$-unitary with idempotent pure partial homomorphism defined by restricting $\sigma\colon T\to G$ to the complement of $I$.

 We now obtain quite easily that the $C^*$-algebra of a strongly $0$-$E$-unitary inverse semigroup $S$ with universal group $G$ and idempotent set $E$ is a partial action cross product $C^*$-algebra for an action of $G$ on $C^*(E)$ (remember we are working in the category of inverse semigroups with $0$).  Indeed, write $S$ as $T/I$ where $T$ is $E$-unitary with maximal group image $G$ and $I$ is an ideal of $T$.  We saw in Lemma~\ref{ideal} that $\wh {E(S)}$ can be identified with the $T$-invariant closed subspace $I^{\perp}$ of $\wh{E(T)}$.    From the definition of the partial action of $G$ on $\wh {E(T)}$, if $g\p$ is defined for $g\in G$ and $\p \in \wh{E(T)}$, then there exists an element of $t\in T$ with $t\p = g\p$.  It follows that $I^{\perp}$  is invariant under the action of $G$, and so in particular, $G$ acts partially on $\wh{E(S)}$ via the identification of $\wh {E(S)}$ with $I^{\perp}$.  Clearly $G\ltimes \wh{E(S)}$ is topologically isomorphic to the reduction of $G\ltimes \wh{E(T)}$ to $I^\perp$.  Thus, by Proposition~\ref{reduction} and Theorem~\ref{main1}, we obtain the following result.

\begin{Thm}\label{main2}
Let $S$ be a countable strongly $0$-$E$-unitary inverse semigroup with universal group $G$ and idempotent set $E$.  Then there is a partial action of $G$ on $\wh{E}$ such that $\mathscr G(S)\cong G\ltimes \wh {E}$; in particular, $\mathscr G(S)$ is Hausdorff.  Consequently,  $C^*(S)\cong C^*(E)\rtimes G\cong C_0(\wh E)\rtimes G$.
\end{Thm}

The corresponding result for reduced algebras also holds. The proof is almost identical to the proof of Theorem \ref{main1reduced}.

\begin{Thm}\label{main2reduced}
Let $S$ be a countable $0$-$E$-unitary inverse semigroup with idempotent set $E$ and maximal group image $G$.  Then $C^*_r(S)\cong C^*(E)\rtimes_r G$.
\end{Thm}

Let $S$ be a countable inverse semigroup with zero and idempotent set $E$.  Recall~\cite{Exel} that the \emph{tight spectrum} $\wh E_{tight}$ of $E$ is the closure in $\wh E$ of the space of ultrafilters (where a filter $F$ is identified with its characteristic function, which is a semi-character, and an ultrafilter is a maximal proper filter).  This space is invariant under the action of $S$.  The associated groupoid of germs is denoted $\mathscr G_{tight}(S)$.  If $S$ is strongly $0$-$E$-unitary with universal group $G$, then the argument above shows that any invariant closed subspace for the action of $S$ is also invariant for $G$.  Thus we immediately obtain the following result.

\begin{Thm}
Let $S$ be a countable strongly $0$-$E$-unitary inverse semigroup $S$ with universal group $G$ and idempotent set $E$.  Then there is a partial action of $G$ on $\wh{E}_{tight}$ and $\mathscr G_{tight}(S)\cong G\ltimes \wh {E}_{tight}$; in particular, $\mathscr G_{tight}(S)$ is Hausdorff.  Consequently,  $C^*(\mathscr G_{tight}(S))\cong C_0(\wh E_{tight})\rtimes G$.
\end{Thm}

Since many of the classical $C^*$-algebras associated to \'etale groupoids come from groupoids of the form $\mathscr G_{tight}(S)$ with $S$ strongly $0$-$E$-unitary~\cite{Exel}, the theorem provides a uniform explanation for many of the partial action cross product results in this context.

\section{Morita equivalence with full cross products}
Let $\theta\colon G\to I_X$ be a partial action of a countable discrete group $G$ on a locally compact Hausdorff space $X$.  Define an equivalence relation on $G\times X$ by $(g,x)\sim (h,y)$ if and only if $x\in X_{g\inv h}$ and $h\inv gx=y$.  The class of $(g,x)$ will be denoted $[g,x]$.  The \emph{enveloping action}~\cite{abadie} or \emph{globalization}~\cite{Lawsonkellendonk} of the action consists of the quotient space $\til X= (G\times X)/{\sim}$ equipped with the $G$-action given by $g'[g,x] = [g'g,x]$.   Abadie proved the following result~\cite{abadie}.

\begin{Thm}\label{abadiesthm}
Suppose that $\theta\colon G\to I_X$ is a partial action of a countable discrete group $G$ on a locally compact Hausdorff space $X$ such that $\til X$ is Hausdorff.  Then $C_0(X)\rtimes G$ is strongly Morita equivalent to the full cross product $C_0(\til X)\rtimes G$.  A similar result holds for reduced $C^*$-algebras.
\end{Thm}

Khoshkam and Skandalis prove a more general result than that of Abadie in~\cite{Skandalis} using different terminology.   Let $G$ be a locally compact group and let $\mathscr G$ be a locally compact groupoid.  By a \emph{continuous cocycle}, they mean a continuous functor $\rho\colon \mathscr G\to G$.  They say the cocycle $\rho$ is:
\begin{itemize}
\item \emph{faithful} if the map $\mathscr G\to \mathscr G_0\times G\times \mathscr G_0$ given by $g\mapsto (r(g),\rho(g),d(g))$ is injective;
\item \emph{closed} if the map $g\mapsto (r(g),\rho(g),d(g))$ is closed;
\item \emph{transverse} if the map $G\times \mathscr G\to G\times \mathscr G_0$ given by $(\gamma,g)\mapsto (\gamma \rho(g),d(g))$ is open.
\end{itemize}

For example, if $G$ is a discrete group acting partially on a locally compact Hausdorff space $X$, then the projection $\rho\colon G\ltimes X\to G$ is a faithful, transverse, continuous cocycle.  The results of Abadie~\cite{abadie} imply that $\rho$ is closed if and only if the enveloping action is Hausdorff.  The implication that $\rho$ closed implies the enveloping action is Hausdorff is also proved in~\cite[Lemma 1.7]{Skandalis}.  It is easy to see that a transverse cocycle comes from a partial action if and only if the map $g\mapsto (\rho(g),d(g))$ is injective, which is stronger than the cocycle being faithful.

Suppose that we are in the more general situation of a faithful, continuous, transverse cocycle $\rho\colon \mathscr G\to G$.  Define the \emph{enveloping action} of $\rho$ as follows.  Let $Y = (G\times \mathscr G_0)/{\sim}$ where $(g,x)\sim (h,y)$ if there exists $\gamma\colon y\to x$ with $\rho(\gamma) = g\inv h$ (note that $\gamma$ is unique by faithfulness if it exists).  Write $[g,x]$ for the class of $(g,x)$ and impose the quotient topology on $Y$.  Notice that $G$ acts continuously on $Y$ by $g[h,x] = [gh,x]$. The following is a combination of Lemma~1.7 and Theorem~1.8 of~\cite{Skandalis}.

\begin{Thm}[Khoshkam and Skandalis]  Let $\rho\colon \mathscr G\to G$ be a faithful, continuous, transverse, closed cocycle.  Then the space $Y$ of the enveloping action is locally compact Hausdorff and there is a Morita equivalence of $\mathscr G$ with $G\ltimes Y$.  Thus $\mathscr G$ is Hausdorff, $C^*(\mathscr G)$ is strongly Morita equivalent to $C_0(Y)\rtimes G$ and $C^*_r(\mathscr G)$ is strongly Morita equivalent to $C_0(Y)\rtimes_r G$.
\end{Thm}

Khoshkam and Skandalis used this theorem to deduce that the $C^*$-algebras of certain $E$-unitary inverse semigroups are strongly Morita equivalent to cross products of a group with commutative $C^*$-algebras.  This has implications for strongly $0$-$E$-unitary inverse semigroups, as we shall see below.  Also their arguments can be made to work more generally.  We develop here the general theory for locally compact groupoids (see~\cite{Paterson} for the precise definition) and then use it to obtain conditions for a locally compact groupoid to be Morita equivalent to a cross product of an inverse semigroup with a commutative $C^*$-algebra.  Restricting to the case of a group recovers the results of Khoshkam and Skandalis.

Recall that if $\mathscr G$ is a topological groupoid, then an \emph{action} of $\mathscr G$ on a space $X$ consists of a continuous map $p\colon X\to \mathscr G_0$ and a continuous function $\theta\colon \mathscr G\times_{\mathscr G_0} X\to X$ (the pullback of $d$ and $p$), written $(g,x)\mapsto gx$, such that
\begin{itemize}
\item $p(x)x=x$;
\item $g(hx) = (gh)x$ if $d(h)=p(x)$ and $r(h)=d(g)$;
\item the diagram
\[\xymatrix{\mathscr G\times_{\mathscr G_0} X\ar[rr]^{\theta}\ar[rd]_{r\pi_1} & & X\ar[ld]^p\\ &\mathscr G_0 &}\] commutes where $\pi_1\colon \mathscr G\times_{\mathscr G_0} X\to \mathscr G$ is the projection.
\end{itemize}
In this case, we say that $X$ is a \emph{$\mathscr G$-space}.  There is an obvious notion of a morphism of $\mathscr G$-spaces.

Given an action $(p,\theta)$ of $\mathscr G$ on a space $X$, we can form the \emph{semidirect product} groupoid given by $\mathscr G\ltimes X = \mathscr G\times_{\mathscr G_0} X$.  Here, the units are the elements of the form $(p(x),x)$ and hence $(\mathscr G\ltimes X)_0$ can be identified with $X$, which we do from now on.  One has $d(g,x) = x$ and $r(g,x) = gx$.  The product is given by \[(g,hy)(h,y) = (gh,y)\] and the inverse by $(g,x)\inv = (g\inv ,gx)$.  It is not hard to verify that if $\mathscr G$ is \'etale, then so is $\mathscr G\ltimes X$.  Of course, this construction agrees with the previous construction of the semidirect product of a group with a space (or transformation groupoid) when $\mathscr G$ is a group.  There is a natural projection $\pi\colon \mathscr G\ltimes X\to \mathscr G$ given by projection to the first coordinate. Notice that if $x\in X$, then $r^{-1}(x)= \{(g,g\inv x)\mid r(g)=p(x)\}$, which is homeomorphic to $r^{-1}(p(x))$ via the projection $\pi$.  From this it is easy to see that if $\{\lambda^e\}_{e\in \mathscr G_0}$ is a left Haar system for $\mathscr G$, then we can define a left Haar system for $\mathscr G\ltimes X$ by putting $\mu^x(B) = \lambda^{p(x)}(\pi(B))$ for a Borel subset $B\subseteq r^{-1}(x)$.  The semidirect product is functorial in the space variable.

We need a few more definitions from topological groupoid theory~\cite{Moerdijkgroupoid,elephant1,elephant2}.  Let $\p\colon \mathscr G\to \mathscr H$ be a continuous functor of topological groupoids.   Then there is a resulting commutative diagram
\begin{equation}\label{definefaithfulfull}
\xymatrix{\mathscr G\ar@/^2pc/[rrd]^{\p}\ar@/_2pc/[ddr]_{r\times d}\ar@{-->}[rd]^{\ov{\p}}  & &\\
                     & \mathscr G_0\times \mathscr G_0\mathop{\times}\limits_{\mathscr H_0\times \mathscr H_0} \mathscr H\ar[r]\ar[d]  & \mathscr H\ar[d]^{r\times d}\\ & \mathscr G_0\times \mathscr G_0\ar[r]_{\p\times \p}  & \mathscr H_0\times \mathscr H_0}
\end{equation}
coming from the universal property of a pullback.  One says that $\p$ is \emph{faithful} if $\ov \p$ is a topological embedding in \eqref{definefaithfulfull}; it is \emph{full} if $\ov \p$ is an open surjection; and it is \emph{fully faithful} if $\ov \p$ is a homeomorphism.

A continuous functor $\p\colon \mathscr G\to \mathscr H$ is called \emph{essentially surjective} if the map $d\pi_2\colon \mathscr G_0\times_{\mathscr H_0} \mathscr H\to \mathscr H_0$ is an open surjection (where the pullback is over the $\p$ and $r$).   A fully faithful, essentially surjective functor is called a \emph{weak equivalence}. Two locally compact groupoids $\mathscr G,\mathscr H$ are said to be \emph{Morita equivalent} if there is a locally compact groupoid $\mathscr K$ and a diagram
\[\xymatrix{\mathscr K\ar[r]^{\psi}\ar[d]_
{\p}& \mathscr H\\ \mathscr G &}\] where $\p$ and $\psi$ are weak equivalences~\cite{elephant1,elephant2}.

A theorem of Renault~\cite{renaultdisintegration} (in the non-Hausdorff case), see also~\cite{muhlywilliams}, implies that if $\mathscr G$ and $\mathscr H$ are Morita equivalent locally compact groupoids, then $C^*(\mathscr G)$ and $C^*(\mathscr H)$ are strongly Morita equivalent, and similarly for $C^*_r(\mathscr G)$ and $C^*_r(\mathscr H)$.

The following is a generalization of the cocycle result of Khoshkam and Skandalis.  The factorization $\p = \pi\alpha$ is the groupoid realization of the hyperconnected-localic factorization of their classifying toposes~\cite{elephant1,elephant2}.

\begin{Thm}\label{cocyclethm}
Let $\p\colon \mathscr G\to \mathscr H$ be a faithful morphism of locally compact groupoids.  Suppose in addition that:
\begin{enumerate}
\item the image of $\ov {\p}$ from \eqref{definefaithfulfull} is closed (or equivalently in light of faithfulness, $\ov{\p}$ is a closed map);
\item the map $\psi\colon \mathscr H\times_{\mathscr H_0}\mathscr G\to  \mathscr H\times_{\mathscr H_0}\mathscr G_0$ given by $\psi(h,g) = (h\p(g),d(g))$ is open;
\item $\mathscr H_0\times_{\mathscr H_0} \mathscr G_0$ is an open subset of $\mathscr H\times_{\mathscr H_0} \mathscr G_0$.  (This is automatic if $\mathscr H$ is \'etale because $\mathscr H_0$ is open in $\mathscr H$.)
\end{enumerate}
Then there is a locally compact Hausdorff space $X$ and an action $(p,\theta)$ of $\mathscr H$ on $X$, called the \emph{enveloping action of $\p$}, such that $\p$ factors as
\[\xymatrix{\mathscr G\ar[rr]^{\alpha}\ar[rd]_{\p} & & \mathscr H\ltimes X\ar[ld]^{\pi}\\ & \mathscr H& }\] with $\alpha$ a weak equivalence.  Thus $\mathscr G$ is Morita equivalent to $\mathscr H\ltimes X$.
\end{Thm}
\begin{proof}
Let $Y=\mathscr H\times_{\mathscr H_0} \mathscr G_0$ (pullback of $d$ and $\p$) and define an equivalence relation on $Y$ by $(h,e)\sim (k,f)$ if there exists $g\colon e\to f$ such that $\p(g) = k\inv h$ (and so in particular $r(h)=r(k)$).  The equivalence class of $(h,e)$ is denoted $[h,e]$.  Let $X=Y/{\sim}$ and let $q\colon X\to Y$ be the quotient map.  Since $Y$ is locally compact, to prove that $X$ is locally compact Hausdorff, it suffices to show that $q$ is an open map and $\sim$ is a closed equivalence relation.

Let $U$ be an open subset of $X$; we must show $q\inv q(U)$ is open.  Let $V = \{(h,g)\in \mathscr H\times_{\mathscr H_0} \mathscr G\mid (h,r(g))\in U\}$; it is open by continuity of $r$.  From the definition of the equivalence relation $\sim$, we have $q\inv q(U) = \psi(V)$, which is open by (2).  Thus $q$ is an open map.  We can define a map \[\beta\colon Y\times Y\to \mathscr G_0\times \mathscr G_0\mathop{\times}\limits_{\mathscr H_0\times \mathscr H_0} \mathscr H\] by $((h,e),(k,f))\mapsto (f,e,k\inv h)$.  The definition of $\sim$ shows that its graph is precisely $\beta\inv (\ov{\p}(\mathscr G))$, and hence $\sim$ is a closed equivalence by (1).  We thus have that $X$ is a locally compact Hausdorff space.

Let us define an action of $\mathscr H$ on $X$.
There is a natural continuous map $p\colon X\to \mathscr H_0$ given by $p([h,e]) = r(h)$. We can define the action map $\theta\colon \mathscr H\times_{\mathscr H_0} X\to X$ by $h[k,e]=[hk,e]$ when $d(h)=r(k)$.  It is a routine exercise to verify that $(p,\theta)$ is indeed an action.  Let $\pi\colon \mathscr H\ltimes X\to \mathscr H$ be the projection.

Define $\alpha\colon \mathscr G\to \mathscr H\ltimes X$ by $\alpha(g) = (\p(g),[\p(d(g)),d(g)])$.  It is immediate that $\p=\pi\alpha$.  It is routine to verify that $\alpha$ is a continuous functor. Let us verify that it is a weak equivalence. First observe that $\alpha|_{\mathscr G_0}\colon \mathscr G_0\to X$ is an open map.  Indeed, as $q\colon Y\to X$ is open, it suffices to check that $e\mapsto (\p(e),e)$ is an open map $\mathscr G_0\to Y$.  But this map is a homeomorphism $\mathscr G_0\to \mathscr H_0\times_{\mathscr H_0} \mathscr G_0$ whose codomain is an open subset of $Y$ by (3).

First we check that $\alpha$ is essentially surjective.  This means that we need that the map $(e,(h,x))\mapsto x$ is an open surjection $\mathscr G_0\times_{X} (\mathscr H\ltimes X)\to X$.  Suppose that $[h,e]\in X$ with $d(h) = \p(e)$.  Then \[(h\inv,[h,e])\colon [h,e]\to [\p(e),e]=\alpha(e).\]  Thus $(e,(h\inv,[h,e]))\mapsto [h,e]$, establishing surjectivity.  Next we prove openness.   Recall that a groupoid is said to be \emph{open} if $d$ is open~\cite{Moerdijkgroupoid}.  Any locally compact groupoid is open due to the existence of a left Haar system.  We claim that if $\rho\colon \mathscr K\to \mathscr L$ is a continuous functor between open groupoids such that $\rho|_{\mathscr K_0}\colon \mathscr K_0\to \mathscr L_0$ is an open map, then $d\pi_2\colon \mathscr K_0\times_{\mathscr L_0} \mathscr L\to \mathscr L_0$ is open.  Indeed, since $d$ is open, it suffices to verify that $\pi_2$ is open.  But $\pi_2$ is open because open maps are stable under pullback.  Applying this to our setting (since $\alpha|_{\mathscr G_0}$ is open), we conclude that $\alpha$ is essentially surjective.

To prove that $\alpha$ is fully faithful, we need to show that $\ov{\alpha}$ defined by \[g\longmapsto (r(g),d(g),\alpha(g))\] is a homeomorphism of $\mathscr G$ with $\mathscr G_0\times \mathscr G_0\times_{X\times X} (\mathscr H\ltimes X)$. Since $\p=\pi\alpha$ is faithful, it follows that $\ov{\alpha}$ is injective and a topological embedding.  It is also surjective since if $(e,f,(h,[k,x]))$ is in the pullback, then $[\p(f),f] = [k,x]$ and $[\p(e),e]=[hk,x]$, whence there exists $g\colon x\to f$ with $k=\p(g)$ and $g'\colon x\to e$ with $\p(g')=hk$. Then $g'g\inv \colon f\to e$ satisfies $\p(g'g\inv) = hkk\inv = h$.  Thus $\ov{\alpha}(g'g\inv) = (e,f,(h,[\p(f),f]))=(e,f,(h,[k,x]))$.  This completes the proof that $\alpha$ is a weak equivalence.
\end{proof}

It is easy to see that~\cite[Theorem~1.8]{Skandalis} for discrete groups is the special case of the above theorem where $\mathscr H$ is a discrete group.  Probably assumption (3) is unnecessary, but we were unable to show that $d\pi_2$ is open without it.

Our goal is to apply Theorem~\ref{cocyclethm} when $\mathscr H=\mathscr G(S)$ for a countable inverse semigroup $S$.    First we want to relate $\mathscr G(S)$-spaces to a class of $S$-spaces.  As usual $E$ denotes the semilattice of idempotents of $S$.  By an \emph{$S$-space}, we mean a pair $(X,\theta)$ where $X$ is a space and $\theta\colon S\to I_X$ is a homomorphism such that $X= \bigcup_{e\in E} X_e$ where $X_e$ is the domain of $\theta(e)$ for $e\in E$.  The $S$-spaces form a category where a morphism $\p\colon X\to Y$ of $S$-spaces is a continuous map such that:
\begin{itemize}
\item $\p\inv (Y_e)=X_e$ for all $e\in E$;
\item $\p(sx) = s\p(x)$ for $x\in X_{s^*s}$.
\end{itemize}

We say that $X$ is a \emph{special} $S$-space if $X_e$ is clopen for all $e\in E$.  For instance, $\wh E$ is a special $S$-space.  The reader is referred to~\cite{Exel} for basic definitions and notions concerning groupoids of germs.

\begin{Thm}\label{equiv}
The categories of $\mathscr G(S)$-spaces and special $S$-spaces are isomorphic.  More, precisely each special $S$-space $X$ has the structure of a $\mathscr G(S)$-space and vice versa.  Moreover, the groupoid of germs $S\ltimes X$ of the action of $S$ on $X$ is topologically isomorphic to the semidirect product $\mathscr G(S)\ltimes X$.
\end{Thm}
\begin{proof}
We just verify the isomorphism at the level of objects; the easy verifications of the details for morphisms are omitted.
First let $X$ be a $\mathscr G(S)$-space; assume that the action is given by $(p,\theta)$ where $p\colon X\to \wh E$.  Define an action of $S$ on $X$ as follows.  Let $X_e = p\inv (D(e))$.  Then since $D(e)$ is clopen in $\wh E$, it follows that $X_e$ is clopen.  Also $X = p\inv (\wh E) = \bigcup_{e\in E} X_e$ since the $D(e)$ cover $\wh E$.  Define, for $s\in S$, a continuous map $\rho_s\colon X_{s^*s}\to X_{ss^*}$ by $\rho_s(x) = [s,p(x)]x$.  This map is well defined since $x\in X_{s^*s}$ implies that $p(x)\in D(s^*s)$ and so $[s,p(x)]\in \mathscr G(S)$.  Continuity is clear from continuity of $p$ and the action of $\mathscr G(S)$.  Also $\rho_{s^*}\rho_s(x) = [s^*,sp(x)][s,p(x)]x=[s^*s,p(x)]x=x$ and so $\rho_s$ and $\rho_{s^*}$ are inverse homeomorphisms.

Next we check that $\rho\colon S\to I_X$ given by $\rho(s)=\rho_s$ is a homomorphism.  First of all we have \[X_{ef} = p\inv (D(ef)) = p\inv (D(e)\cap D(f)) = p\inv (D(e))\cap p\inv (D(f))= X_e\cap X_f.\] Secondly, if $s\leq t$ and $x\in X_{s^*s}$, then $p(x)\in D(s^*s)\subseteq D(t^*t)$ and so $[s,p(x)]=[t,p(x)]$, whence $\rho_s\leq \rho_t$.  Thus $\rho$ is order preserving.
Thus to check that $\rho$ is a homomorphism, we just need to show that $\rho_s\rho_t=\rho_{st}$ whenever $s^*s=tt^*$ (cf.~\cite[Chapter~3, Theorem~5]{Lawson}).  In this case, $\rho_s\rho_t(x) = \rho_s([t,p(x)]x)= [s,tp(x)][t,p(x)]x = [st,p(x)]x = \rho_{st}(x)$.  Thus $\rho$ is a special action.

We construct an isomorphism $S\ltimes X\to \mathscr G(S)\ltimes X$ by $[s,x]\mapsto ([s,p(x)],x)$ with inverse $([s,p(x)],x)\mapsto [s,x]$.  It is routine to verify that these maps are inverse continuous functors.

Next suppose that $\rho\colon S\to I_X$ gives a special action of $S$ on $X$.  Define $p\colon X\to \wh E$ by \[p(x)(e) = \begin{cases} 1 & x\in X_e\\ 0 & x\notin X_e\end{cases}.\]  It is easy to see that $p(x)\in \wh E$.  The fact that $p(x)$ is a semilattice homomorphism is trivial since $X_e\cap X_f=X_{ef}$.  It is non-zero exactly because $x\in X_e$ for some $e\in E$.  To show that $p$ is continuous, it suffices to show that $p\inv(D(e))$ is clopen for all $e\in E$. But $p(x)\in D(e)$ if and only if $x\in X_e$ and $X_e$ is clopen by definition of a special action.  Define the action of $\mathscr G(S)$ on $X$ by $[s,p(x)]x = sx$.  We must show that this is well defined.  First of all $p(x)\in D(s^*s)$ implies that $x\in X_{s^*s}$ and so $sx$ is defined.  Also if $[t,p(x)]=[s,p(x)]$, then we can find $u\leq s,t$ with $p(x)\in D(u^*u)$.  But then $x\in X_{u^*u}$ and so $sx=ux=tx$.  Thus the action is well defined.  Let us verify continuity.  Suppose $p(x)\in D(s^*s)$ and
$U$ is a neighborhood of $sx$.  Then we can find a neighborhood $V\subseteq X_{s^*s}$ so that $x\in V$ and $sV\subseteq U$.  Then $(\mathscr G(S)\times_{\wh E} X)\cap ((s,D(s^*s))\times V)$ is a neighborhood of $([s,p(x)],x)$ mapped by the action into $U$.

We omit the straightforward verification that these two constructions are inverse to each other.
\end{proof}

The reader should observe that Paterson's universal property of $\mathscr G(S)$~\cite{Paterson} is an immediate consequence of this result.

Let $\mathscr G$ be a locally compact groupoid and $S$ a countable inverse semigroup.  We define a \emph{continuous, faithful, transverse, closed cocycle} $\p\colon \mathscr G\to S$ to be a continuous faithful functor $\p\colon \mathscr G\to \mathscr G(S)$ satisfying (1) and (2) of Theorem~\ref{cocyclethm} (condition (3) being automatic as $\mathscr G(S)$ is \'etale). The following result generalizes~\cite[Theorem~1.8]{Skandalis} to inverse semigroups.  If $S$ is an inverse semigroup acting on a locally compact Hausdorff space $X$, then one can form the cross-product $C^*$-algebra $C_0(X)\rtimes S$~\cite[Section~9]{Exel}.  Moreover, $C_0(X)\rtimes S\cong C^*(S\ltimes X)$ by~\cite[Theorem~9.8]{Exel}.  Exel does not speak of reduced inverse semigroup cross-products, but one would presume the analogous result holds.  Theorems~\ref{cocyclethm} and~\ref{equiv} have the following corollary.

\begin{Cor}\label{invcocyclecor}
Let $\rho\colon \mathscr G\to S$  be a continuous, faithful, transverse, closed cocycle where $\mathscr G$ is a locally compact groupoid and $S$ is a countable inverse semigroup.  Then there is a locally compact Hausdorff space $X$ equipped with an action by $S$ so that $\mathscr G$ is Morita equivalent to the groupoid of germs $S\ltimes X$.  Consequently, $C^*(\mathscr G)$ is strongly Morita equivalent to $C_0(X)\rtimes S$.  An analogous result holds for reduced $C^*$-algebras if $S$ is a group.
\end{Cor}

Our next goal is to apply this result to inverse semigroups, in particular to strongly $0$-$E$-unitary inverse semigroups.  First we address a question that does not seem to have been satisfactorily answered in the literature: the functoriality of the construction $S\mapsto \mathscr G(S)$. If $\p\colon E\to F$ is a semilattice homomorphism, then there is obviously a continuous map $\p\colon \wh F\to \wh E$ induced by precomposition with $\p$.  However, it turns out that under certain circumstances there is also a continuous map $\wh E\to \wh F$.

A map of topological spaces $f\colon X\to Y$ is said to be \emph{coherent} if, for each quasi-compact open subset $U$ of $Y$, one has that $f\inv(U)$ is quasi-compact open~\cite{johnstone}.  It is natural to say that $f$ is \emph{locally coherent} if, for each $x\in X$, there is a neighborhood $U$ of $x$ so that $f|_U\colon U\to Y$ is coherent.

A poset $P$ is naturally a $T_0$ topological space via the \emph{Alexandrov topology}.  The open sets are the downsets where we recall that a \emph{downset}  is a subset $X\subseteq P$ such that $y\leq x$ and $x\in X$ implies $y\in X$.  The continuous maps between posets are precisely the order preserving ones.  The quasi-compact open subsets are easily verified to be the finitely generated downsets and so a map of posets $f\colon P\to Q$ is coherent if and only if the preimage of any finitely generated downset is finitely generated.  If $p\in P$, then $p^{\downarrow}$ denotes the \emph{principal downset} generated by $p$.

Recall that a \emph{filter} $F$ on a poset $P$ is a non-empty upset (defined dually to downsets) such that any two elements of $F$ have a common lower bound.   For a semilattice, filters are precisely the non-empty subsemigroups which are upsets.  In general, if $B$ is a non-empty subsemilattice of $E$, then the up-closure $B^{\uparrow} = \{x\in E\mid \exists f\in B, x\geq f\}$ is a filter.  It is easy to see that the semi-characters of a semilattice $E$ are in bijection with the filters via the correspondence $\p\mapsto \p\inv(1)$ and $F\mapsto \chi_F$.  Thus we can identify $\wh E$ with the space of filters on $E$.  The compact open $D(e)$ corresponds to those filters containing $e$.

Let $\p\colon E_1\to E_2$ be a semilattice homomorphism.  We can define a map $\wh \p\colon \wh E_1\to \wh E_2$ by $\wh \p(F) = \p(F)^{\uparrow}$.  This is well defined since $\p(F)$ is a non-empty subsemilattice.  The question is: when is the map $\wh \p$ continuous?  The answer when $E$ has a maximum is well known~\cite{johnstone} and the adaptation for the general case is not difficult.

\begin{Prop}\label{continuity}
Let $\p\colon E_1\to E_2$ be a semilattice homomorphism.  Then $\wh \p\colon \wh E_1\to \wh E_2$ is continuous if and only if $\p$ is locally coherent.
\end{Prop}
\begin{proof}
Suppose first that $\p$ is locally coherent.  To establish continuity it suffices to show that $\wh \p\inv (D(e))$ is clopen for $e\in E_2$.  We claim that \[\wh \p\inv (D(e)) = \bigcup_{f\in \pinv(e^{\downarrow})}D(f),\] and hence is open.  Indeed, if $\p(f)\leq e$ and $f\in F$ where $F$ is a filter, then $e\in \p(F)^{\uparrow}=\wh \p(F)$ and so $\wh \p(F)\in D(e)$.  Conversely, if $\wh \p(F)\in D(e)$, then there exists $x\in \p(F)$ with $e\geq x$.  If $x=\p(f)$ with $f\in F$, then $f\in \pinv(e^{\downarrow})$ and so $F\in D(f)$.  Note that so far we have not used local coherence.

Since $\wh E_1$ is covered by the compact open sets $D(x)$ with $x\in E_1$, to establish that $\wh \p\inv (D(e))$ is closed it suffices to show that $D(x)\cap \wh \p\inv(D(e))$ is closed for each $x\in E_1$.  That is, it suffices to show that $\bigcup_{f\in \pinv(e^{\downarrow})\cap x^{\downarrow}}D(f)$ is closed.  Let $X$ be a downset of $E_1$ containing $x$ such that $\p|_X$ is coherent.  Then there exists $x_1,\ldots, x_n\in X\cap \p\inv(e^{\downarrow})$ such that $f\in X\cap \pinv(e^{\downarrow})$ if and only if $f\leq x_i$ for some $i$.  So if $f\in \pinv(e^{\downarrow})\cap x^{\downarrow}$, then one has $f=fx\leq x_ix$ for some $i$.  Also $x_jx\in \pinv(e^{\downarrow})\cap x^{\downarrow}$ for all $j$.  As $f\leq x_ix$ implies $D(f)\subseteq D(x_ix)$ it follows that \[\bigcup_{f\in \pinv(e^{\downarrow})\cap x^{\downarrow}}D(f)= \bigcup_{i=1}^nD(x_ix)\] and hence is closed. This yields the continuity of $\wh \p$.

Suppose next that $\wh \p$ is continuous.  We claim that, for each $x\in E_1$, one has that $\p|_{x^{\downarrow}}\colon x^{\downarrow}\to E_2$ is coherent.  Clearly it suffices to show that if $e\in E_2$, then $\p\inv (e^{\downarrow})\cap x^{\downarrow}$ is finitely generated as a downset.  By continuity $\wh \p\inv (D(e))$ is a clopen subset of $\wh E_1$ and hence $K=D(x)\cap \wh \p\inv (D(e))$ is compact.  The arguments in the previous paragraph yield \[K = \bigcup_{f\in \pinv(e^{\downarrow})\cap x^{\downarrow}} D(f)\] and so there exist $x_1,\ldots, x_n\in \pinv(e^{\downarrow})\cap x^{\downarrow}$ such that $f\in \pinv(e^{\downarrow})\cap x^{\downarrow}$ implies $D(f)\subseteq D(x_1)\cup\cdots \cup D(x_n)$.  Let $f^{\uparrow}$ be the principal filter generated by $f$.  If $f^{\uparrow}\in D(x_i)$, then $f\leq x_i$.  Thus $x_1,\ldots, x_n$ generate $\pinv(e^{\downarrow})\cap x^{\downarrow}$ as a downset.  We conclude that $\p$ is locally coherent.
\end{proof}

The proof of Proposition~\ref{continuity} shows that $\p\colon E_1\to E_2$ is locally coherent if and only if $\p|_{e^\downarrow}\colon e^\downarrow\to E_2$ is coherent for each $e\in E_1$.  We shall use this later.

Note that local coherence is not automatic.  Let $E_1$ consist of a top $1$, a bottom $0$ and an infinite anti-chain $X$ and let $E_2=\{0,1\}$.  Let $\p\colon E_1\to E_2$ send the top of $E_1$ to $1$ and all remaining elements to $0$.  Since $E_1$ has a top, the remark above shows that in order for $\p$ to be locally coherent, it must be coherent.  But $\p\inv (0^{\downarrow})$ is not finitely generated as a downset.

\begin{Prop}\label{locallycoherentrestricts}
Let $\p\colon S\to T$ be a locally coherent homomorphism of inverse semigroups.  Then $\p|_{E(S)}\colon E(S)\to E(T)$ is locally coherent.
\end{Prop}
\begin{proof}
Put $\psi=\p|_{E(S)}$ and suppose that $\p$ is locally coherent. Let $e\in E(S)$ and choose a downset $X$ of $S$ containing $e$ so that $\p|_X\colon X\to T$ is coherent.  Let $f\in E(T)$ and suppose that $s_1,\ldots, s_m$ generate $X\cap \pinv(f^{\downarrow})$ as a downset.  Set $e_i =es_i^*s_i$.  Since $\p(s_i)\leq f$, we have $\p(e_i)\leq \p(s_i^*s_i)\leq f$ and so the $e_i$ belong to $\psi\inv(f^{\downarrow})\cap e^{\downarrow}$.  If $x\in \psi\inv(f^{\downarrow})\cap e^{\downarrow}$, then $x\leq e$ implies $x\in X$.  Therefore, $x\leq s_i$ for some $i$.  But then $x\leq e_i$.  It follows that $\psi\inv(f^{\downarrow})\cap e^{\downarrow}$ is generated by $e_1,\ldots, e_m$ and so $\psi$ is locally coherent.
\end{proof}

The converse of the above proposition is false.  Let $E_1$ be the semilattice constructed before the proposition and let $T=E_1\cup \{z\}$ where $z^2=1$ and $za=a=az$ for $a\in X\cup \{0\}$.  Then $E_1=E(T)$ and hence the inclusion $\iota$ clearly satisfies $\iota\colon E_1\to E(T)$ is locally coherent.  However, $\iota\colon E_1\to T$ is not locally coherent.  If it were, then since $E_1$ has a maximum, it would have to be coherent.  But it is not since $\iota\inv(z^{\downarrow}) = X\cup \{0\}$, which is not finitely generated as a downset.  This example can be generalized, using~\cite[Theorem~5.17]{discretegroupoid} to show that if $T$ is an inverse semigroup such that $\mathscr G(T)$ is not Hausdorff, then the inclusion $\iota\colon E(T)\to T$ satisfies $\iota\colon E(T)\to E(T)$ is locally coherent but $\iota\colon E(T)\to T$ is not.

Our next goal is to show that Paterson's universal groupoid construction is functorial if one restricts to inverse semigroup morphisms $\p\colon S\to T$ such that $\p|_{E(S)}$ is locally coherent.

\begin{Thm}\label{functorialityofG}
Let $\p\colon S\to T$ be a homomorphism of inverse semigroups such that $\p|_{E(S)}$ is locally coherent.  Then there is a continuous homomorphism $\Phi\colon \mathscr G(S)\to \mathscr G(T)$ given by $[s,F]\mapsto [\p(s),\wh \p(F)]$ where $F$ is a filter on $E(S)$ with $s^*s\in F$.
\end{Thm}
\begin{proof}
Let us first show that $\Phi$ is well defined.  If $[s,F]=[t,F]$, then there exists $u\leq s,t$ with $u^*u\in F$.  Then $\p(u)\leq \p(s),\p(t)$ and $\p(u)^*\p(u)\in \p(F)^{\uparrow}=\wh \p(F)$.  Thus $[\p(s),\wh \p(F)]=[\p(t),\wh \p(F)]$.  To verify continuity, let $[s,F]\in \mathscr G(S)$ and suppose that $(t,U)$ is a basic neighborhood of $[\p(s),\wh \p(F)]$.  Then we can find $u\leq \p(s),t$ such that $u^*u\in \wh \p(F)$.  Thus there exists $e\in F$ with $\p(e)\leq u^*u$.  Consider the neighborhood \[W=(s,D(s^*s)\cap D(e)\cap \wh \p\inv (U))\] of $[s,F]$ in $\mathscr G(S)$.  If $[s,F']\in W$, then $\Phi([s,F']) = [\p(s),\wh\p(F')]$ satisfies $\wh\p(F')\in U$.  Also $e\in F'$ implies that $u^*u\in \wh\p(F')$.  It follows that $[\p(s),\wh \p(F')]= [t,\wh \p(F')]\in (t,U)$.  Thus $\Phi$ is continuous.

To verify that $\Phi$ is a functor, first observe that on objects, $\Phi(F) = \wh \p(F)$ and so $\Phi(d([s,F])) = \wh\p(F)=d([\p(s),\wh\p(F)])=d(\Phi([s,F]))$. On the other hand, $\Phi(r([s,F])) = \wh \p(sF)$, whereas $r(\Phi([s,F])) = s\wh \p(F)$.  Thus we must show that $\wh{\p}(sF)=\p(s)\wh \p(F)$.

If $e\in \wh \p(sF)$, then there exists $x\in sF$ with $\p(x)\leq e$.  But then $s^*xs\in F$ and $\p(s^*xs)\leq \p(s)^*e\p(s)$.  This shows that $\p(s)^*e\p(s)\in \wh \p(F)$ and hence $e\in \p(s)\wh \p(F)$.  Conversely, if $e\in \p(s)\wh \p(F)$, then $\p(s)^*e\p(s)\in \wh \p(F)$ and so there exists $x\in F$ with $\p(x)\leq \p(s)^*e\p(s)$. But then $e\geq \p(sxs^*)$ and $s^*(sxs^*)s\in F$ because $s^*s\in F$.  Thus $sxs^*\in sF$ and so $e\in \wh \p(sF)$.  We conclude $\p(s)\wh \p(F)=\wh \p(sF)$.

Then remaining verification that $\Phi$ is a functor is the computation
\begin{align*}
\Phi([s',sF])\Phi([s,F]) &= [\p(s'),\wh{\p}(sF)][\p(s),\wh \p(F)]= [\p(s's),\wh \p(F)]\\ &= \Phi([s's,F]).
\end{align*}
This completes the proof.
\end{proof}

We introduce a condition, called the \emph{Khoshkam-Skandalis} condition (or \emph{KS condition} for short), on an inverse semigroup homomorphism $\p\colon S\to T$ that guarantees that Theorem~\ref{cocyclethm} applies to the induced homomorphism $\Phi\colon \mathscr G(S)\to \mathscr G(T)$.

\begin{Def}[KS condition]
Let $\p\colon S\to T$ be an inverse semigroup homomorphism.  Then $\p$ is said to satisfy the \emph{KS condition} if, for all $e,f\in E(S)$, one has $\p|_{eSf}\colon eSf\to T$ is coherent.
\end{Def}

The KS condition on $\p$ implies that it is locally coherent and hence the restriction of $\p$ to $E(S)$ is locally coherent by Proposition~\ref{locallycoherentrestricts}.  Thus $\Phi\colon \mathscr G(S)\to \mathscr G(T)$ can be defined as per Theorem~\ref{functorialityofG}.  The KS condition was considered by Khoshkam and Skandalis for the special case of the maximal group image homomorphism $\sigma\colon S\to G$.

In the paper~\cite{InvExp}, a semigroup homomorphism $\p\colon S\to T$ was defined to be an \emph{$F$-morphism} if $\p\inv (t)$ has a maximum element for each $t\in T$.  For instance, $S$ is an \emph{$F$-inverse monoid} if and only if $\sigma\colon S\to G$ is an $F$-morphism.  The original motivation for considering $F$-morphisms was to study partial actions of inverse semigroups. We claim that an $F$-morphism satisfies the KS condition.

\begin{Prop}
An $F$-morphism satisfies the KS condition.
\end{Prop}
\begin{proof}
Let $\p\colon S\to T$ be an $F$-morphism and suppose $e,f\in E(S)$.  Let $t\in T$ and suppose $u$ is the maximum element in $\pinv (t)$.  Then we have $\pinv(t^{\downarrow})\cap eSf = (euf)^{\downarrow}$.  Indeed, $\p(euf)\leq t$ and $euf\in eSf$.  If $s\in eSf$ and $\p(s)\leq t$, then $s\leq u$ and so $s=esf\leq euf$.  This shows that $\p|_{eSf}$ is coherent.
\end{proof}

We now prove a series of lemmas that will lead to our main result on Morita equivalence.

\begin{Lemma}\label{closed}
Let $\p\colon S\to T$ be an inverse semigroup homomorphism satisfying the KS condition.  Abusing notation, we write $\p\colon \mathscr G(S)\to \mathscr G(T)$ for the induced morphism.   Let
\begin{equation}\label{uselesslabel}
\psi\colon \mathscr G(S)\to \mathscr G(S)_0\times \mathscr G(S)_0\times \mathscr G(T)
\end{equation}
be given by $\psi(g) = (r(g),d(g),\p(g))$.  Then $\psi$ is a closed map.
\end{Lemma}
\begin{proof}
Let $X\subseteq \mathscr G(S)$ be closed.  Since the sets of the form \[D(e)\times D(f)\times (t,D(t^*t))\] with $e,f\in E(S)$ and $t\in T$ form a cover of the right hand side of \eqref{uselesslabel} by compact open sets, it suffices to show $\psi(X)\cap (D(e)\times D(f)\times (t,D(t^*t)))$ is closed in $D(e)\times D(f)\times (t,D(t^*t))$ for all choices of $e,f,t$.  Let $s_1,\ldots,s_n$ be a finite generating set for $eSf\cap \pinv(t^{\downarrow})$.

\begin{claim}
One has that
\[\psi(X)\cap (D(e)\times D(f)\times (t,D(t^*t))) = \psi\left(X\cap \left(\bigcup_{i=1}^n(s_i,D(s_i^*s_i))\right)\right)\] holds.
\end{claim}
\begin{proof}[Proof of claim]
Suppose that $[s,F]\in X\cap (s_i,D(s_i^*s_i))$.  Without loss of generality, we may assume $s\leq s_i$.  Then $s^*s\leq s_i^*s_i\leq f$.  Thus $F\in D(f)$.  Similarly, $ss^*\leq s_is_i^*\leq e$ and so $s^*es\geq  s^*s\in F$, yielding $e\in sF$.  Thus to show that $\psi([s,F])$ belongs to the left hand side, it remains to show that $[\p(s),\wh \p (F)]\in (t,D(t^*t))$.  First observe that $\p(s)\leq \p(s_i)\leq t$.  Also $t^*t\geq \p(s^*s)\in \p(F)$ and so $t^*t\in \wh\p (F)$.  Thus $[\p(s),\wh \p(F)] = [t,\wh \p(F)]\in (t, D(t^*t))$.  We conclude $\psi([s,F])$ belongs to the left hand side.

Next suppose that $[s,F]\in X$ with $\psi([s,F])$ in the left hand side. Then $f\in F$ and $e\in sF$, whence $s^*es\in F$.  Thus $(esf)^*esf =fs^*esf\in F$.  Therefore $[s,F]=[esf,F]$ and so we may assume without loss of generality that $s\in eSf$.  Next, since $[\p(s),\wh \p (F)]\in (t,D(t^*t))$, it follows that there exists $u\in T$ with $u\leq \p(s),t$ and $u^*u\in \wh \p(F)$, which in turn means we can find $x\in F$ with $\p(x)\leq u^*u$.  Then $(sx)^*(sx)=xs^*s\in F$ and so $[s,F]=[sx,F]$.  Also $\p(sx)\leq \p(s)u^*u=u\leq t$.  Thus without loss of generality, we may assume $\p(s)\leq t$.  Therefore, there exists $i$ with $s\leq s_i$.  Then $[s,F]=[s_i,F]$ and hence $[s,F]\in (s_i,D(s_i^*s_i))$.  Thus $\psi([s,f])$ belongs to the right hand side.
\end{proof}
The desired result follows from the claim because the right hand side is compact and hence the left hand side is as well.  But then the left hand side is closed in the compact space $D(e)\times D(f)\times (t,D(t^*t))$.
\end{proof}

Our next lemma establishes that condition (2) of Theorem~\ref{cocyclethm} is always fulfilled in the case of a groupoid morphism induced by an inverse semigroup homomorphism.

\begin{Lemma}\label{transverse}
Suppose that $\p\colon S\to T$ is an inverse semigroup homomorphism such that $\p|_{E(S)}$ is locally coherent.  Again, we use $\p\colon \mathscr G(S)\to \mathscr G(T)$ for the induced morphism.  Let \[\psi\colon \mathscr G(T)\times _{\mathscr G(T)_0}\mathscr G(S)\to \mathscr G(T)\times _{\mathscr G(T)_0}\mathscr G(S)_0\] be given by $(h,g)\mapsto (h\p(g),d(g))$.  Then $\psi$ is open.
\end{Lemma}
\begin{proof}
Let $X= \mathscr G(T)\times _{\mathscr G(T)_0}\mathscr G(S)$ and $Y=\mathscr G(T)\times _{\mathscr G(T)_0}\mathscr G(S)_0$.  Then
 a typical element of $X$ is of the form $([t,\p(s)\wh \p(F)],[s,F])$.  It is easy to see that a basic neighborhood of such a point is of the form $W=((t,V)\times (s,U))\cap X$ with $U\subseteq D(s^*s)$ and $V\subseteq D(t^*t)\cap D(\p(ss^*))$.  We claim the image of $W$ is $N=((t\p(s),\p(s)^*V)\times U)\cap Y$ and hence open.

 If $([t,\p(s)\wh \p(F')],[s,F'])\in W$, then its image is $([t\p(s),\wh \p(F')],F')\in N$ since $\p(s)\wh \p(F')\in V$ implies $\wh\p(F')\in \p(s)^*V$.  Conversely, given an element $([t\p(s),\wh\p(F')],F')\in N$, we have $[s,F']\in (s,U)$ and $\p(s)\wh\p(F')\in \p(ss^*)V = V$. Thus $([t,\p(s)\wh\p(F')],[s,F'])\in W$ with image $([t\p(s),\wh\p(F')],F')$.

 This proves that $\psi$ is open.
\end{proof}

A crucial notion from inverse semigroup theory~\cite{Lawson} is that of an idempotent pure homomorphism.  An inverse semigroup homomorphism $\p\colon S\to T$ is said to be \emph{idem\-po\-tent pure} if $\p\inv (E(T))=E(S)$, that is, $\p(s)\in E(T)$ implies $s\in E(S)$.  For instance, an inverse semigroup $S$ is $E$-unitary if and only if the maximal group image homomorphism $\sigma\colon S\to G$ is idempotent pure.  Let us say that $\p\colon S\to T$ is \emph{locally idempotent pure} if $\p|_{eSe}\colon eSe\to T$ is idempotent pure for each $e\in E(S)$.  An inverse semigroup is said to be \emph{locally $E$-unitary} if $eSe$ is $E$-unitary for each $e\in E(S)$.  Clearly, $S$ is locally $E$-unitary if and only if $\sigma\colon S\to G(S)$ is locally idempotent pure.  It turns out that being locally idempotent pure is enough to guarantee that the corresponding morphism of groupoids is faithful, at least if we put aside topological concerns.

\begin{Lemma}\label{faithful}
Let $\p\colon S\to T$ be a locally idempotent pure homomorphism of inverse semigroups such that $\p|_{E(S)}$ is locally coherent.  Denote also by $\p$ the induced morphism $\mathscr G(S)\to \mathscr G(T)$.  Then the map $\psi$ given by \[g\mapsto (r(g),d(g),\p(g))\] is injective.  In particular, if $\p$ satisfies the KS condition, then the induced morphism of groupoids is faithful.
\end{Lemma}
\begin{proof}
Suppose that $\psi([s,F]) = \psi([t,F'])$.  Then $F=F'$, $sF=tF$ and $[\p(s),\wh \p(F)]=[\p(t),\wh \p(F)]$.  We need to find $u\in S$ such that $u^*u\in F$ and $u\leq s,t$.  We can find $v\leq \p(s),\p(t)$ with $v^*v\in \wh\p (F)$.  Hence there exists $x\in F$ with $\p(x)\leq v^*v$.  Then $(sx)^*(sx),(tx)^*(tx)\in F$ and so $[s,F]=[sx,F]$ and $[t,F]=[tx,F]$.  Moreover, $\p(sx) =\p(s)v^*v\p(x)=v\p(x)=\p(t)v^*v\p(x)=\p(tx)$.  Thus, replacing $s$ by $sx$ and $t$ by $tx$, we may assume that $\p(s)=\p(t)$.

Let $f=tt^*st^*$.  Then $f\in tt^*Stt^*$ and, because $\p(s)=\p(t)$, we have $\p(f) = \p(tt^*)\p(s)\p(t)^*=\p(tt^*)\in E(T)$.  Thus $f\in E(S)$ since $\p$ is locally idempotent pure and hence $u=ft\leq t$.  But $u=tt^*st^*t\leq s$. It remains to prove $u^*u\in F$.  First observe that $u^*u= t^*ts^*tt^*st^*t= t^*ts^*tt^*s$.   Applying $sF=tF$ and $t^*t\in F$ yields $tt^*\in tF=sF$ and hence $s^*tt^*s\in F$.  Thus $u^*u=t^*ts^*tt^*s\in F$.  This establishes $[s,F]=[t,F]$ and so $\psi$ is injective.

The last statement follows since Lemma~\ref{closed} shows that $\psi$ is a closed mapping and hence a topological embedding, being injective.
\end{proof}

Putting together all these lemmas, we have that if $\p\colon S\to T$ is a locally idempotent pure homomorphism of inverse semigroups satisfying the KS condition, then $\p\colon \mathscr G(S)\to \mathscr G(T)$ is a continuous, faithful, transverse, closed cocycle.  Thus we have, by an application of Corollary~\ref{invcocyclecor}, the following theorem, which is one of the main results of this paper.

\begin{Thm}\label{mainnewstuff}
Let $\p\colon S\to T$ be a locally idempotent pure homomorphism of countable inverse semigroups satisfying the Khoshkam-Skandalis condition that $\p|_{eSf}$ is coherent for all $e,f\in E(S)$.  Then there is a locally compact Hausdorff space $X$ acted on by $T$ such that $\mathscr G(S)$ is Morita equivalent to the germ groupoid $T\ltimes X$.  Consequently,  $C^*(S)$ is strongly Morita equivalent to $C_0(X)\rtimes T$.  Moreover, if $T$ is a group, then $C^*_r(S)$ is strongly Morita equivalent to $C_0(X)\rtimes_r T$.
\end{Thm}

Our first special case generalizes~\cite[Theorem~3.10]{Skandalis} and also covers~\cite[Example~3.12(b)]{Skandalis}, where a certain locally $E$-unitary inverse semigroup that is not $E$-unitary is considered.

\begin{Cor}\label{EunitKS}
Let $S$ be a countable locally $E$-unitary inverse semigroup with maximal group image homomorphism $\sigma\colon S\to G$.  Suppose that, for all $e,f\in E(S)$ and $g\in G$, one has that $eSf\cap \sigma\inv (g)$ is finitely generated as a downset.  Then there is a locally compact Hausdorff space $X$ and an action of $G$ on $X$ so that $C^*(S)$ is strongly Morita equivalent to $C_0(X)\rtimes G$ and $C^*_r(S)$ is strongly Morita equivalent to $C_0(X)\rtimes_r G$.
\end{Cor}

In the $E$-unitary case, the above condition is equivalent to the enveloping action being Hausdorff by~\cite[Proposition~3.9]{Skandalis}.

Recall that if $S$ is a strongly $0$-$E$-unitary inverse semigroup with universal group partial homomorphism $\theta\colon S\to \mathscr U(S)$, then $S\cong T/I$ where $T=\{(s,\theta(s))\mid s\neq 0\}\cup (\{0\}\times \mathscr U(S))$ and $I= \{0\}\times \mathscr U(S)$.  Moreover, $T$ is $E$-unitary.  It is easy to see that $T$ satisfies the conditions of Corollary~\ref{EunitKS} if and only if $(eSf\setminus \{0\})\cap \theta\inv (g)$ is finitely generated as a downset for all $e,f\in E(S)$ and $g\in \mathscr U(S)$.  Also it is easy to see that if $\gamma\colon \mathscr G\to G$ is a continuous, faithful, transverse, closed cocycle and $X$ is a closed subset of $\mathscr G_0$, then $\gamma\colon \mathscr G|_X\to G$ is also a continuous, faithful, transverse, closed cocycle.  We thus have the following corollary of our previous work.

\begin{Cor}\label{stronglymoritacase}
Let $S$ be a countable strongly $0$-E-unitary inverse semigroup and suppose that the universal group partial homomorphism $\iota\colon S\to \mathscr U(S)$ satisfies $eSf\cap \iota\inv (g)$ is finitely generated as a downset for all $e,f\in E(S)\setminus \{0\}$ and $g\in \mathscr U(S)$.  Then there is an action of $\mathscr U(S)$ on a locally compact Hausdorff space $X$ such that $C^*(S)$ is strongly Morita equivalent to $C_0(X)\rtimes \mathscr U(S)$ and $C^*_r(S)$ is strongly Morita equivalent to $C_0(X)\rtimes_r \mathscr U(S)$.  The corresponding result holds for the tight $C^*$-algebra of $S$ in the sense of Exel~\cite{Exel}.
\end{Cor}

One can imitate the argument of~\cite[Proposition~3.9]{Skandalis} to show that the above condition is equivalent to the enveloping action being Hausdorff.

For example, if $|X|\geq 2$ and $P_X$ is the polycyclic inverse monoid, then $\iota\colon P_X\to F_X$ is the map taking $uw^*$ to the reduced form of $uw^{-1}$.  It is clear that the image of $\iota$ consists of all reduced words $v$ that can be written as a positive word multiplied by the inverse of a positive word.  The unique maximal element in $\iota\inv(v)$ is $v$ itself.  It follows immediately that Corollary~\ref{stronglymoritacase} applies to $P_X$.  More generally, call an inverse semigroup $S$ with zero \emph{strongly $0$-$F$-inverse} if $eSf\cap \iota\inv (g)$ has a maximum element, when not empty, for all $g\in \mathscr U(S)$. For an inverse monoid, this is equivalent to asking that $\iota\inv (g)$ have a maximum element whenever it is not empty. For instance, $P_X$ is strongly $0$-$F$-inverse, as are the graph inverse semigroups of~\cite{Paterson}.

\begin{Cor}
Let $S$ be a countable strongly $0$-$F$-inverse semigroup with universal group $\mathscr U(S)$.   Then there is an action of $\mathscr U(S)$ on a locally compact Hausdorff space $X$ such that $C^*(S)$ is strongly Morita equivalent to $C_0(X)\rtimes \mathscr U(S)$ and $C^*_r(S)$ is strongly Morita equivalent to $C_0(X)\rtimes_r \mathscr U(S)$.  The corresponding result holds for the tight $C^*$-algebra of $S$ in the sense of Exel~\cite{Exel}.
\end{Cor}

%\bibliographystyle{abbrv}
%\bibliography{standard2}

\end{document}